\documentclass[12pt,a4paper]{amsart}

\usepackage[T1]{fontenc}
\usepackage{newtxtext}
\usepackage{newtxmath}
\usepackage{microtype}

\usepackage[margin=1in]{geometry}

\usepackage{titlesec}
\titleformat{\section}{\normalfont\large\bfseries}{\thesection}{1em}{}
\titleformat{\subsection}{\normalfont\normalsize\bfseries}{\thesubsection}{1em}{}

\usepackage{amsmath}

\usepackage{amsthm}
\usepackage{amsfonts}
\newcommand{\gbin}[2]{\genfrac{[}{]}{0pt}{}{#1}{#2}}
\usepackage{mathtools}
\usepackage{bm}

\usepackage[hidelinks]{hyperref}
\usepackage[nameinlink]{cleveref}

\usepackage{comment}

\DeclareMathOperator{\Cov}{Cov}
\DeclareMathOperator{\Supp}{Supp}
\DeclareMathOperator{\supp}{supp}
\DeclareMathOperator{\wt}{wt}

\newcommand{\F}{\mathbb{F}}

\newcommand{\R}{\mathbb{R}}

\newcommand{\Harm}{\mathrm{Harm}}

\newcommand{\qbinom}[2]{%
  \genfrac{[}{]}{0pt}{}{#1}{#2}_{q}
}

\theoremstyle{plain}

\newtheorem{theorem}{Theorem}[section]
\newtheorem{lemma}[theorem]{Lemma}
\newtheorem{proposition}[theorem]{Proposition}
\newtheorem{corollary}[theorem]{Corollary}

\theoremstyle{definition}

\newtheorem{definition}[theorem]{Definition}
\newtheorem{example}[theorem]{Example}

\theoremstyle{remark}

\newtheorem{remark}[theorem]{Remark}
\begin{document}

\title[Simonis' Approach of MacWilliams Identity]{Harmonic Higher Weight distributions, Simonis' Approach of MacWilliams Identity and Moments}

\author[Chakraborty]{Himadri Shekhar Chakraborty}

\address
{
	Department of Mathematics, 
	Shahjalal University of Science and Technology\\ Sylhet-3114, Bangladesh\\
}
\email{himadri-mat@sust.edu}

\author[Tanver]{Mehedi Hasan Tanver*}
\thanks{*Corresponding author}
\address
{
	Department of Mathematics, 
	Shahjalal University of Science and Technology\\ Sylhet-3114, 
	Bangladesh\\
}
\email{mhtanver13@gmail.com}

\date{\today}



\begin{abstract}
We present a combinatorial proof of Simonis type MacWilliams identity for harmonic higher weight distributions of linear codes. 
Furthermore, we investigate the statistical moments of the harmonic
higher weight enumerators for random linear codes. Defining the enumerators via rank functions of the generator matrices of linear codes, we prove that its expectation vanishes for all non-trivial harmonic functions due to the inherent symmetry of random matrices, and we also derive an explicit, non-trivial formula for the covariance.
\end{abstract}

\maketitle

\section{Introduction}

\noindent The weight distribution of a linear code~(cf.~\cite{macwilliams1977}) is a fundamental invariant that reveals much about its combinatorial and geometric structure . The MacWilliams identity~\cite{macwilliams1963} establishes a celebrated algebraic relationship between the weight distribution of a linear code and that of its dual. This identity was later generalized by Wei \cite{wei1991}, who introduced the notion of generalized Hamming weights (or higher weights) to characterize the support weight distribution of subcodes of a given dimension. Subsequently, Kløve \cite{klove1992} and Simonis \cite{simonis1994} derived MacWilliams-type identities for these higher weight distributions. While Kløve's approach relied on extension fields, Simonis provided an elementary proof using a purely combinatorial counting argument over the underlying finite field.
Of late, \"{O}zen and Tekin \cite{ozen2013} defined the higher
weight enumerators of linear codes over finite fields via rank function of its generator matrix. Moreover, they presented the expectation and the covariance formulas for the higher weight distributions of linear codes.
\medskip

\noindent In a different direction, Delsarte \cite{delsarte1978} introduced the theory of discrete harmonic functions on finite sets, which was later applied to coding theory by Bachoc \cite{bachoc1999} to present a
harmonic generalization of the MacWilliams identity. This harmonic
generalization plays an important role to provide an alternative
simple proof of the Assmus-Mattson theorem~\cite{assmus1969} for combinatorial $t$-designs associated to linear codes. Recently, Britz, Chakraborty, and Miezaki \cite{britz2024} successfully merged these two domains, higher weights and harmonic functions, by defining the harmonic higher weight enumerator and proving a MacWilliams-type identity for it. However, their derivation mostly depends on matroid theory~(cf.~\cite{oxley1992}) and associated algebraic concepts, including harmonic analogue of Tutte polynomials \cite{chakraborty2023, tutte1954} and Greene-type relation \cite{greene1976}; see also~\cite{BrChMiIsTa2024}.

\medskip
\noindent The first goal of this paper is to provide an alternative, elementary proof of the MacWilliams identity for harmonic higher weight distributions by adapting Simonis'~\cite{simonis1994} combinatorial
method to the harmonic setting. The central combinatorial challenge in this adaptation is forcing the harmonic sums to factor appropriately; we resolve this by using a simple reflection identity of the harmonic function.

\medskip
\noindent The second goal of this paper is to investigate the statistical moments of the harmonic higher weight enumerator for random linear codes, extending the work of \"{O}zen and Tekin \cite{ozen2013}. We compute the expectation and the covariance of the harmonic weight enumerators. We prove that the expectation (first moment) of the harmonic higher 
weight enumerators is zero for all non-trivial harmonic functions. 
This occurs because random codes are, on average, perfectly symmetric, mirroring the properties of the trivial code. While this first moment vanishes, we show that the covariance (second moment) is strictly non-zero, providing a non-trivial statistical invariant that can be used to verify the MacWilliams identity in a probabilistic setting.

\medskip
\noindent The rest of the paper is organized as follows. In Section 2, we lay the groundwork by establishing the notation and definitions for linear codes, puncturing and shortening, and discrete harmonic functions. In Section 3, we prove Simonis type MacWilliams identity (Theorem~\ref{thm:main_theorem})
for harmonic higher weight distributions. We also provide explicit examples that verify the identity. 
In Section 4, we investigate the statistical moments of the harmonic higher weight enumerators for random linear codes. By establishing a relationship between the harmonic higher weight enumerator and harmonic rank function (Theorem~\ref{prop:rank}), we prove the vanishing of the expectation (Theorem~\ref{thm:expectation}) and derive an explicit formula for the covariance (Theorem~\ref{thm:cov}). Finally, in Section 5, we provide explicit examples verifying these probabilistic results.

\section{Preliminaries}

\noindent In this section, we establish the notation and definitions for linear codes, puncturing and shortening, and discrete harmonic functions. We largely follow the notation established in \cite{britz2024}.

\subsection{Linear codes}

\noindent Let $E := \{1, 2, \dots, n\}$ be a finite set. 
Let $\F_q$ be the finite field of order~$q$, where $q$ is a prime power.
The $n$-dimensional vector space over $\F_q$, denoted by $\F_q^n$, 
is equipped with the usual inner product $u \cdot v = \sum_{i=1}^n u_i v_i$ for $u, v \in \F_q^n$. We call the elements of~$\F_{q}^{n}$ as
\emph{vectors}.
For a vector $u \in \F_q^n$, the \textit{support} of $u$ is defined as $\supp(u) := \{i \in E \mid u_i \neq 0\}$, and its \textit{weight} is $\wt(u) := |\supp(u)|$. This notion extends to any subset $D \subseteq \F_q^n$:
\begin{align*} 
	\Supp(D) & := \bigcup_{u \in D} \supp(u),\\ 
	\wt(D) &:= |\Supp(D)|. 
\end{align*}
We refer to $\wt(D)$ as the \textit{effective length} of $D$.

An $\F_q$-\emph{linear code} of length $n$ 
is a linear subspace of $\F_q^n$. An $\F_q$-linear code of length~$n$ is called $[n,k]$ \emph{code} if its dimension is $k$. The \emph{dual code} 
of an $\F_q$-linear code $C$, denoted by $C^\perp$, is defined as  
\[
	C^\perp 
	:= 
	\{ v \in \F_q^n \mid u \cdot v = 0 \text{ for all } u \in C \}.
\] 
It is well-known that if $C$ is an $[n,k]$ code, then $C^\perp$ is an $[n, n-k]$ code.

Let $C$ be an $[n,k]$ code.
Let $r$ be a non-negative integer such that $r \le k$. We define $\mathcal{D}_r(C)$ as the set of all $r$-dimensional subcodes of $C$. The $r$-\emph{th higher weight distribution} of $C$ is the sequence $\{A_i^{(r)}(C) \mid 0 \le i \le n\}$, where:
\[
 	A_i^{(r)}(C) 
 	:= 
 	\#\{D \in \mathcal{D}_r(C) \mid \wt(D) = i\},
\]
and satisfies the following MacWilliams-type identity
given by Simonis~\cite[{Theorem 1}]{simonis1994}:
\[
	\sum_{i=0}^{n} 
	\binom{n-i}{n-j} 
	A_{i}^{(r)}(C^\perp) 
	= 
	\sum_{u=0}^r 
	q^{u(j-k+u-r)} 
	\qbinom{j-k}{r-u}
	\sum_{v=0}^{n} 
	\binom{n-v}{j} 
	A_{v}^{(u)}(C),
\]
for $0\leq i,j\leq n$.
\noindent
Here, for all integers $a,b$ and $b \geq 0$, 
the \textit{Gaussian binomial coefficient} $\qbinom{a}{b}$
is defined as follows: 
\[
	\qbinom{a}{b}
	:=
	\begin{cases}
		\prod\limits_{i=0}^{b-1}\,\dfrac{ q^{a}-q^{i}}{q^{b}-q^{i}}
		& \text{for } b\neq 0,\\
		1 &\text{for } b=0.
	\end{cases}	
\]
For $a,b\geq 0$, $\qbinom{a}{b}$ counts the number of $b$-dimensional subspaces of an $a$-dimensional linear space over $\F_q$. Moreover, $\binom{a}{b}$ denotes the usual binomial coefficient.

Let $X \subseteq E$. Then for a vector $u \in \F_q^n$,  the restriction of $u$ to $X$ is a vector 
\[
	u_X : X \to \F_{q}; \quad i \mapsto u_i
\] 
in $\F_q^{|X|}$.
Also for an $\F_q$-linear code~$C$, 
we define the \textit{punctured code} $C^X$ and the \textit{shortened code} $C(X)$ as follows:
\begin{align*}
C^X &:= \{u_X \in \F_q^{|X|} \mid u \in C\}, \\
C(X) &:= \{u_X \in \F_q^{|X|} \mid u \in C,\, \supp(u) \subseteq X\}.
\end{align*}

The following useful dimension formulas that relates puncturing and shortening of a code are immediately from~\cite[Proposition 1]{simonis1994}.

\begin{proposition}\label{prop:dim_formulas}
Let $C$ be an $[n, k]$ code over $\F_q$. For any subset $X \subseteq E$:
\begin{enumerate}
    \item[(i)] $\dim C(X) = k - \dim C^{E \setminus X}$.
    \item[(ii)] $\dim C^X + \dim C^\perp(X) = |X|$.
\end{enumerate}
\end{proposition}

\subsection{Discrete Harmonic Functions}

Let $E_d := \{X \subseteq E \mid |X| = d\}$ for $d = 0, 1, \dots, n$.
The set of all subsets of $E$ is denoted by $2^{E}$.
We denote by 
$\R 2^{E}$ and $\R E_{d}$
the real vector spaces spanned by the elements of  
$2^{E}$ and $E_{d}$,
respectively. 

An element of 
$\R E_{d}$
is denoted by
\begin{equation}\label{Equ:FunREd}
	f :=
	\sum_{Z \in E_{d}}
	f(Z) Z,
\end{equation}
and is identified with the real-valued function on 
$E_{d}$
given by 
$Z \mapsto f(Z)$.  
Such an element $f \in {\R E_d}$ can be extended to a function $\widetilde{f} \in {\R2^E}$ by setting, for all $X \subseteq E$:
\begin{equation}\label{Equ:tildef}
	\widetilde{f}(X) 
	:= 
	\sum_{\substack{Z \in E_d \\ Z \subseteq X}} f(Z). 
\end{equation}
{Note that $\widetilde{f}(\emptyset) = f(\emptyset)$ when $d=0$,
	and that $\widetilde{f}(\emptyset) = 0$ otherwise}. 
If an element 
$g \in \R 2^{E}$
is equal to $\widetilde{f}$  
for some $f \in \R E_{d}$, 
then we say that $g$ has degree~$d$.

\noindent
The \textit{differentiation operator} $\gamma$ on ${\R E_d}$ is defined by linearity from the identity:
\begin{equation}\label{Equ:Gamma}
	\gamma(Z) 
	:= 
	\sum_{\substack{Y \in E_{d-1} \\ Y \subseteq Z}} Y 
\end{equation}
for all $Z \in E_d$. The space of \textit{discrete harmonic functions of degree $d$} is the kernel of $\gamma$:
\[ 
	\Harm_d(n) 
	:= 
	\ker\left(\gamma \big|_{\mathbb{R}{E_d}}\right). 
\]

\begin{remark}[\cite{bachoc1999, delsarte1978}]\label{Rem:Gamma}
	Let $f \in \Harm_{d}(n)$ and $i \in \{0,1,\ldots,d-1\}$.
	Then 
	$\gamma^{d-i}(f) = 0$. 
	This means from definition~(\ref{Equ:Gamma})
	that
	$$\sum_{X \in E_{i}}\left(\sum_{\substack{Z \in E_{d}, X \subseteq Z}} f(Z)\right) X = 0.$$
	This implies that $\sum_{\substack{Z \in E_{d}, X \subseteq Z}} f(Z) = 0$
	for any $X \in E_{i}$.
	In particular, if $X\in E_{d-1}$, then 
	\begin{equation}\label{eq:harm_condition}
		\sum_{e \in E \setminus X} f(X \cup \{e\}) = 0.
	\end{equation}

\end{remark}

\begin{remark}\label{Rem:New}
	Let~$f \in \Harm_{d}(n)$. 
	Since $\sum_{Z \in E_{d}} f(Z) = 0$, 
	it is easy to check from~(\ref{Equ:Gamma}) that

	$\sum_{X \in E_{t}} \widetilde{f}(X) = 0$, where
	$1 \leq d \leq t \leq n$.
\end{remark}

\begin{remark}\label{Rem:BachocLem}
	From the definition of $\widetilde{f}$ 
	for $f \in \Harm_{d}(n)$,
	we have $\widetilde{f}(X) = 0$
	for each $X \subseteq {E}$ such that $|X| < d$. 
	
\end{remark}

The following useful identity of discrete harmonic functions 
is immediate from~\cite[Lemma~2.1]{bachoc1999}. We prefer to 
call this identity as \emph{reflection identity}. 
We prefer to call this identity as \emph{reflection identity}.
For the proof, we refer the reader to~\cite[Remark 3.4]{britz2024}.

\begin{lemma}[Reflection Identity]\label{prop:reflection}
	For $f \in \Harm_d(n)$ and any subset $X \subseteq E$:
	\[ \widetilde{f}(E \setminus X) = (-1)^d \widetilde{f}(X). \]
\end{lemma}

\begin{remark}\label{Rem:nminusd}
	From the above equality it is clear that if $|X| > n-d$, 
	then $\widetilde{f}(X) = 0$.
\end{remark}

\begin{example}\label{ex:harm_func}
Let $E = \{1, 2, 3, 4\}$ and $d = 2$. Define $f \in {\R E_2}$ by:
\[
	\begin{array}{ccc}
		f(\{1,2\}) = a_1,&
		f(\{1,3\}) = a_2,& 
		f(\{1,4\}) = a_3,\\
		f(\{2,3\}) = a_4, &
		f(\{2,4\}) = a_5, &
		f(\{3,4\}) = a_6.
	\end{array}	
\]
If $X = \{2,3,4\}$,
then $\widetilde{f}(X) = a_4+a_5+a_6$.
Applying differential operator $\gamma$ on $f$, it follows that
$f \in \Harm_2(4)$ if and only if 
\[
	\sum_{e\in E\setminus X}
	f(X\cup \{e\})
	= 0
\]
for all $X \in E_1$.

This implies
$a_{3} = a_{4} = -(a_{1}+a_{2})$, $a_{5} = a_{2}$ and $a_{6} = a_{1}$.

\end{example}

\section{Harmonic analogue of Simonis' MacWilliams identity}

In this section, we provide a combinatorial proof of the MacWilliams identity for higher weight distribution
associated to harmonic function by adapting the counting method that was originally introduced by Simonis \cite{simonis1994} for the classical higher weights. The central harmonic ingredient required for this adaptation is the reflection identity.

Let $C$ be an $[n, k]$ code over $\F_q$, and let $f \in \Harm_d(n)$ with $d \neq 0$.

\begin{definition}\label{def:harm_weight_enum}
	The \emph{harmonic $r$-th higher weight enumerator} of $C$ associated to $f$ is defined as:
	\[ W_{C,f}^{(r)}(z) := \sum_{i=0}^{n} A_{i,f}^{(r)}(C) z^{n-i}, \]
	where the harmonic $r$-th higher weight distribution is given by:
	\[ A_{i,f}^{(r)}(C) := \sum_{\substack{D \in \mathcal{D}_r(C) \\ \wt(D) = i}} \widetilde{f}(\Supp(D)). \]
\end{definition}

Then it is immediate from Remarks~\ref{Rem:BachocLem} and~\ref{Rem:nminusd}
that
the harmonic $r$-th higher weight enumerator
$W_{C,f}^{(r)}(z)$ of~$C$ 
associated to $f\in \Harm_{d}(n)$ is always divisible by $z^d$. We therefore write the polynomial as:
\[
	W_{C,f}^{(r)}(z) = z^d Z_{C,f}^{(r)}(z),
\]
where $Z_{C,f}^{(r)}(z)$ is a polynomial of degree~$n-2d$
of the form:

\[ 
Z_{C,f}^{(r)}(z) 
:= 
\sum_{i=d}^{n-d} 
A_{i,f}^{(r)}(C) 
z^{n-i-d}. 
\]

\begin{remark}
	Clearly,
	$A_{0,f}^{(r)}(C) = 0$ for all $0 \leq r \leq k$.
\end{remark}

\begin{remark}

	Since every $1$-dimensional subspace of $C$ contains $q-1$ nonzero codewords, 
	$(q - 1)A_{i,f}^{(1)}(C) = A_{i,f}(C)$ for $0 \leq i \leq n$,
	where 
	\[
	A_{i,f}(C)
	:= 
	\sum_{{u} \in C, \wt({u}) = i} 
	\widetilde{f}(\supp({u}))
	\]
	is the harmonic weight distribution of~$C$.

\end{remark}

\begin{remark}
	If $\deg f = 0$, then
	{$W_{C,f}^{(r)}(z) = f(\emptyset)\, W_{C}^{(r)}(z)$},
	where $W_{C}^{(r)}(z)$ denotes
	the $r$-\emph{th higher weight enumerator} of $C$:
	\[ W_C^{(r)}(z) := \sum_{i=0}^n A_i^{(r)}(C) z^{n-i}. \]

\end{remark}

\begin{remark}
	Because $\widetilde{f}(E) = 0$ for any $f \in \Harm_d(n)$ with $d \ge 1$, we have $A_{n,f}^{(r)}(C) = 0$. 
	
\end{remark}

Let $C$ be an $[n,k]$ code, and $f\in \Harm_{d}(n)$. For integers $j, m \ge 0$, we define the \textit{harmonic $N$-distribution} of~$C$ as:
\[ N_{j,f}^m(C) := \sum_{\substack{X \subseteq E,\, |X| = j \\ \dim C(X) = m}} \widetilde{f}(X). \]

Now we state the following useful lemmas. The third one is immediately
from~\cite[Lemma 3]{simonis1994}.

\begin{lemma}\label{lem:harm_lemma1}
For $f \in \mathrm{Harm}_d(n)$ and integers $j$ ($d\leq j\leq n-d$) 
and $r$ ($0\leq r\leq k$):
\[
  \sum_{m=0}^{k} \gbin{m}{r}_{\!q} N^m_{j,f}(C)
  = \sum_{i=d}^{n-d} \binom{n-i-d}{n-j-d} A^{(r)}_{i,f}(C).
\]
\end{lemma}

\begin{proof}
Let $\mathcal{A}^r = \bigcup_{i=d}^{n-d} \mathcal{D}_r(C)$ be the set of all
$r$-dimensional subcodes of $C$ with effective length between $d$ and $n-d$.

Let
\[
  \Sigma := \sum_{\substack{(D,Y)\colon D \in \mathcal{A}^r,\;
    Y \in \binom{E}{n-j} \\ \mathrm{Supp}(D) \cap Y = \emptyset}}
  \widetilde{f}(Y)
\]
Now evaluate the sum in two ways.

\noindent\textit{Count by $Y$.}
For each fixed $Y \in \binom{E}{n-j}$, the subcodes $D$ with
$\mathrm{Supp}(D) \cap Y = \emptyset$ are exactly the $r$-dimensional subcodes
of $C(E \setminus Y)$, of which there are $\bigl[\frac{\dim C(E \setminus Y)}{r}\bigr]_q$.
Setting $X = E \setminus Y$ and applying Proposition~\ref{prop:reflection}:
\begin{equation}\label{Equ:Sigma1}
  \Sigma
  = \sum_{X \in \binom{E}{j}} \widetilde{f}(E \setminus X)
    \gbin{\dim C(X)}{r}_{\!q}
  = (-1)^d \sum_{m=0}^{k} \gbin{m}{r}_{\!q} N^m_{j,f}(C).
\end{equation}

\noindent\textit{Count by $D$.}
Fix $D \in \mathcal{A}^r_i$ with $\mathrm{Supp}(D) = I$, $|I| = i$.
The valid $Y$'s satisfy $Y \subseteq E \setminus I$, $|Y| = n-j$:
\[
  \sum_{\substack{Y \in \binom{E}{n-j} \\ Y \subseteq E \setminus I}}
  \widetilde{f}(Y)
  = \binom{n-i-d}{n-j-d} \widetilde{f}(E \setminus I)
  = (-1)^d \binom{n-i-d}{n-j-d} \widetilde{f}(I),
\]
where the first step counts, for each $Z \in \binom{E \setminus I}{d}$,
the $\binom{n-i-d}{n-j-d}$ sets $Y \supseteq Z$ of size $n-j$ in $E \setminus I$,
and the second step applies Proposition~\ref{prop:reflection}.
Summing over all $D$:
\begin{equation}\label{Equ:Sigma2}
  \Sigma = (-1)^d \sum_{i=d}^{n-d} \binom{n-i-d}{n-j-d} A^{(r)}_{i,f}(C).
\end{equation}

\noindent Equating (\ref{Equ:Sigma1}) and (\ref{Equ:Sigma2}) and dividing by $(-1)^d$ gives the result.
\end{proof}

\begin{lemma}\label{lem:harm_lemma2}
For $f \in \Harm_d(n)$ and all integers $j$ and $m$:
\[ N_{j,f}^m(C^\perp) = (-1)^d N_{n-j, f}^{m+k-j}(C). \]
\end{lemma}

\begin{proof}
The statement obviously holds for $j \notin \{0,1,\ldots,n\}$.
Now let $X \subseteq E$ be a subset such that $|X| = j$ and $\dim C^\perp(X) = m$. Define $Y = E \setminus X$, so $|Y| = n - j$. By Proposition \ref{prop:dim_formulas}(ii), we have 
\[ \dim C^X = |X| - \dim C^\perp(X) = j - m. \]
By Proposition \ref{prop:dim_formulas}(i), we have 
\[ \dim C(Y) = \dim C(E \setminus X) = k - \dim C^X = k - (j - m) = m + k - j. \]
This shows that the complement map $\phi(X) = E \setminus X$ sends subsets of size $j$ with $\dim C^\perp(X) = m$ to subsets of size $n-j$ with $\dim C(Y) = m+k-j$.

To verify that $\phi$ is a bijection, we must show that every valid $Y$ maps back to a valid $X$. Let $Y \subseteq E$ be any subset such that $|Y| = n-j$ and $\dim C(Y) = m+k-j$. Let $X = E \setminus Y$. Then $|X| = n - |Y| = j$. Furthermore, by Proposition \ref{prop:dim_formulas}(i), $\dim C(Y) = k - \dim C^X$, which implies $\dim C^X = k - (m+k-j) = j-m$. Applying Proposition \ref{prop:dim_formulas}(ii), we obtain 
\[ \dim C^\perp(X) = |X| - \dim C^X = j - (j-m) = m. \]
Thus, $X$ satisfies the exact dimension conditions for the dual code. Since taking the complement twice returns the original set (i.e., $E \setminus (E \setminus X) = X$), the map $\phi$ is a bijection between these two families of subsets.
    
Applying this bijection and the reflection identity (Lemma~\ref{prop:reflection}), we compute the dual distribution:
\begin{align*}
    N_{j,f}^m(C^\perp) &= \sum_{\substack{X \subseteq E,\, |X|=j \\ \dim C^\perp(X)=m}} \widetilde{f}(X) \\
    &= \sum_{\substack{Y \subseteq E \\ |Y|=n-j \\ \dim C(Y)=m+k-j}} \widetilde{f}(E \setminus Y) \\
    &= (-1)^d \sum_{\substack{Y \subseteq E \\ |Y|=n-j \\ \dim C(Y)=m+k-j}} \widetilde{f}(Y) \\
    &= (-1)^d N_{n-j, f}^{m+k-j}(C).
\end{align*}
This completes the proof.
\end{proof}

\begin{lemma}
\label{lem:q_vandermonde}
For all integers $a, b$ and $r \ge 0$, we have
\[ 
	\qbinom{a+b}{r} 
	= \sum_{u=0}^r 
	q^{u(a+u-r)} 
	\qbinom{a}{r-u} 
	\qbinom{b}{u}. 
\]
\end{lemma}

We now assemble the lemmas to prove the MacWilliams identity
for harmonic $r$-th higher weight distribution.

\begin{theorem}[Harmonic MacWilliams Identity]\label{thm:main_theorem}
	Let $C$ be an $[n,k]$ code. Let $f \in \Harm_d(n)$ and 
	$0\leq r\leq k$.
	Then for all integers $j$ ($d\leq j\leq n-d$):
\[
\sum_{i=d}^{n-d} 
\binom{n-i-d}{n-j-d} 
A_{i,f}^{(r)}(C^\perp) 
= \sum_{u=0}^r 
(-1)^d q^{u(j-k+u-r)} \qbinom{j-k}{r-u} 
\sum_{v=d}^{n-d} 
\binom{n-v-d}{j-d} 
A_{v,f}^{(u)}(C).
\]

\end{theorem}

\begin{proof}
We begin by applying Lemma \ref{lem:harm_lemma1} to the dual code $C^\perp$:
\begin{equation}\label{Equ:MacLem} 
	\sum_{i=d}^{n-d} 
	\binom{n-i-d}{n-j-d} 
	A_{i,f}^{(r)}(C^\perp) 
	= 
	\sum_{m=0}^k \qbinom{m}{r} N_{j,f}^m(C^\perp). 
\end{equation}
Applying Lemma \ref{lem:harm_lemma2} to the right-hand side 
and substituting $s = m + k - j$ (so $m = s - k + j$):

\begin{align*} 
	\sum_{m=0}^k \qbinom{m}{r} N_{j,f}^m(C^\perp) 
	& = 
	(-1)^d \sum_{m=0}^k \qbinom{m}{r} N_{n-j, f}^{m+k-j}(C)\\
	& =
	(-1)^d \sum_{s=0}^k \qbinom{s+j-k}{r} N_{n-j, f}^s(C). 
\end{align*}
By Lemma \ref{lem:q_vandermonde} with $a = j-k$ and $b = s$:
\[ 
	\qbinom{s+j-k}{r} 
	= 
	\sum_{u=0}^r q^{u(j-k+u-r)} 
	\qbinom{j-k}{r-u} 
	\qbinom{s}{u}. 
\]
Substituting this expansion explicitly into our sum yields:
\begin{align*}
	\sum_{m=0}^k \qbinom{m}{r} N_{j,f}^m(C^\perp) 
	& = 
    (-1)^d \sum_{s=0}^k \qbinom{s+j-k}{r} N_{n-j, f}^s(C) \\
    &= (-1)^d \sum_{s=0}^k \left( \sum_{u=0}^r q^{u(j-k+u-r)} \qbinom{j-k}{r-u} \qbinom{s}{u} \right) N_{n-j, f}^s(C) \\
    &= (-1)^d \sum_{u=0}^r q^{u(j-k+u-r)} \qbinom{j-k}{r-u} \sum_{s=0}^k \qbinom{s}{u} N_{n-j, f}^s(C).
\end{align*}
Finally, we apply Lemma \ref{lem:harm_lemma1} with $j$ replaced by $n-j$ and $r$ replaced by $u$. Note that $n - (n-j) - d = j - d$.
Then 

\[ \sum_{s=0}^k \qbinom{s}{u} N_{n-j, f}^s(C) = \sum_{v=d}^{n-d} \binom{n-v-d}{j-d} A_{v,f}^{(u)}(C). \]
Substituting this explicitly into the previous equation
and from Eq.~(\ref{Equ:MacLem}) we conclude:

\[ \sum_{i=d}^{n-d} \binom{n-i-d}{n-j-d} A_{i,f}^{(r)}(C^\perp) = (-1)^d \sum_{u=0}^r q^{u(j-k+u-r)} \qbinom{j-k}{r-u} \sum_{v=d}^{n-d} \binom{n-v-d}{j-d} A_{v,f}^{(u)}(C). \qedhere\]

\end{proof}

The following remarks concerning harmonic generalization of the
MacWilliams identity for $r$-th higher weight distributions stated in Theorem~\ref{thm:main_theorem} might be helpful:

\begin{remark}\label{rem:simonis_remark5}
	Let $\mathbf{A}_f^{(u)}(C) = (A_{v,f}^{(u)}(C))_{v \in \mathcal{I}_d}^T$, where
	$\mathcal{I}_d = \{d, d+1, \dots, n-d\}$.
    The left-hand side of Theorem \ref{thm:main_theorem} is a linear system in the dual harmonic distribution ${\mathbf{A}}_f^{(r)}(C^\perp)$ with co-efficient matrix $\mathbf{B}_{d} = \left( \binom{n-i-d}{n-j-d} \right)_{i,j \in \mathcal{I}_d}$. 
 
    It is immediate that $\mathbf{B}_{d}$ is equivalent to a lower-triangular matrix of size $(n-2d+1) \times (n-2d+1)$ with ones on the diagonal. Thus, $\det(\mathbf{B}_{d}) = 1 \neq 0$. 
    Therefore, we can solve explicitly for the dual harmonic distributions:
    \[
    	{\mathbf{A}}_f^{(r)}(C^\perp) 
    	= 
    	\sum_{u=0}^r 
    	\mathbf{M}_{u,d}^r \, \mathbf{A}_f^{(u)}(C),
    \]
    where the entries of the $(n-2d+1) \times (n-2d+1)$ transformation matrix $\mathbf{M}_{u,d}^r$ are explicitly given by the formula:
    \begin{equation}\label{Equ:Mud}
    	\left(\mathbf{M}_{u,d}^r\right)_{i,v} 
    	:= 
    	(-1)^d 
    	\sum_{j=d}^{n-d} 
    	(-1)^{i+j} 
    	\binom{n-j-d}{n-i-d} 
    	q^{u(j-k+u-r)} 
    	\qbinom{j-k}{r-u} 
    	\binom{n-v-d}{j-d},
    \end{equation}
    for $i, v \in \mathcal{I}_d$. 

\end{remark}

\begin{remark}
    If $u=r$ in Eq.~(\ref{Equ:Mud}) the Gaussian binomial coefficient $\qbinom{j-k}{r-u} = \qbinom{j-k}{0} = 1$. 
    Moreover, by~\cite[Theorem 15(ii)]{macwilliams1977},
    the matrix $\mathbf{M}_{r,d}^r$ simplifies as follows:
    \begin{align*}
        \left(\mathbf{M}_{r,d}^r\right)_{i,v}
        & =
        (-1)^{d}
        \sum_{j=d}^{n-d} 
        (-1)^{i+j} 
        \binom{n-j-d}{n-i-d} 
        q^{r(j-k)} 
        \binom{n-v-d}{j-d} \\
        & =
        (-1)^{d} q^{-r(k-d)} \sum_{j=d}^{n-d} (-1)^{i+j} \binom{n-j-d}{n-i-d} q^{r(j-d)} \binom{n-v-d}{j-d}\\
        & =
        (-1)^d 
        q^{-r(k-d)} 
        K_{i-d}(v-d; n-2d, q^r),
    \end{align*}
    where $K_i(v; N, q^r)$ is the \emph{Krawtchouk polynomial} 
	defined by:
   \[
   		K_i(v; N, q^r) 
   		:= 
   		\sum_{m=0}^i (-1)^m (q^r-1)^{i-m} 
   		\binom{v}{m} \binom{N-v}{i-m}.
   \]
\end{remark}
 
\begin{remark}
	Taking $r=1$ and $d\neq 0$ reduces Theorem~\ref{thm:main_theorem} to Bachoc's MacWilliams identity~\cite{bachoc1999} for harmonic weight distributions. 

\end{remark}  

\begin{example}\label{ex:macwilliams_verify}
We verify Theorem \ref{thm:main_theorem} on the $[5, 2]$ code $C$ over $\F_2$ generated by
\[
G = \begin{pmatrix} 1 & 1 & 1 & 0 & 0 \\ 0 & 0 & 0 & 1 & 1 \end{pmatrix},
\]
as presented in \cite[Example 4.1]{britz2024}. Let $d=1$ and define $f \in \Harm_1(5)$ by
\[
f(\{1\})=a, \quad f(\{2\})=b, \quad f(\{3\})=c, \quad f(\{4\})=d, \quad f(\{5\})=-(a+b+c+d).
\]
The non-zero codewords of $C$ have supports $\{1,2,3\}, \{4,5\}$, and $E$, yielding 
$$A_{3,f}^{(1)}(C) = a+b+c, \quad A_{2,f}^{(1)}(C) = -(a+b+c).$$ 
The dual code $C^\perp$ has dimension $3$, and its seven $2$-dimensional subcodes yield 
$$A_{3,f}^{(2)}(C^\perp) = a+b+c, \quad A_{4,f}^{(2)}(C^\perp) = -(a+b+c).$$

We evaluate the identity for $j=3$ and $r=2$. With $n=5, k=2$, and $d=1$, we have $j-k=1$. The sum over the dual distribution evaluates to
\[
\sum_{i=1}^4 \binom{5-i-1}{5-3-1} A_{i,f}^{(2)}(C^\perp) = \binom{1}{1}A_{3,f}^{(2)}(C^\perp) + \binom{0}{1}A_{4,f}^{(2)}(C^\perp) = a+b+c.
\]
To verify that this equals the corresponding sum over $C$, we expand the sum over $u \in \{0, 1, 2\}$. The term for $u=0$ vanishes since $\qbinom{1}{2} = 0$, and the term for $u=2$ vanishes because $A_{v,f}^{(2)}(C) = 0$ for all $v$ (as $\dim C = 2$ and $\widetilde{f}(E) = 0$). Thus for $u=1$, we have
\[
2^0 \qbinom{1}{1} \left[ \binom{4-3}{2}A_{3,f}^{(1)}(C) + \binom{4-2}{2}A_{2,f}^{(1)}(C) \right] = 1 \cdot \left[ 0 - (a+b+c) \right] = -(a+b+c).
\]
Multiplying this by $(-1)^1$ yields $a+b+c$. which confirms the identity.
\end{example}

\section{Moments of the harmonic higher weight enumerator}

\noindent In this section, we investigate the statistical moments of the harmonic higher weight enumerator for random linear codes, extending the probabilistic framework of Özen and Tekin \cite{ozen2013}. Let $C$ be a random $[n, k]$ code. Since $C$ is random, its harmonic weight distribution $A_{i,f}^{(r)}(C)$ consists of random variables, making the enumerator $W_{C,f}^{(r)}(z)$ a random polynomial. We analyze its statistical moments using standard probability theory.

\begin{definition}\label{def:moments}
Let $X$ be a discrete random variable with values~$x$ and 
probabilities $\Pr(X=x)$. Then the \emph{expectation} of $X$, denoted $\mathbb{E}[X]$, is defined as
\[ \mathbb{E}[X] = \sum_{x} x \Pr(X = x). \]
The \emph{covariance} of discrete random variables $X$ and $Y$ 
is defined as
\[ \Cov(X, Y) := \mathbb{E}[(X - \mathbb{E}[X])(Y - \mathbb{E}[Y])] = \mathbb{E}[XY] - \mathbb{E}[X]\mathbb{E}[Y]. \]
\end{definition}

By linearity of expectation, the expected value of the polynomial $W_{C,f}^{(r)}(z)$ is the sum of the expected values of its coefficients:
\[
\mathbb{E}\left[ W_{C,f}^{(r)}(z) \right] := \sum_{i=d}^{n-d} \mathbb{E}\left[ A_{i,f}^{(r)}(C) \right] z^{n-i}.
\]
Similarly, the covariance between two specializations $W_{C,f}^{(r)}(x)$ and $W_{C,f}^{(r)}(y)$ is given by the standard formula:
\[
\Cov\left( W_{C,f}^{(r)}(x), W_{C,f}^{(r)}(y) \right) := \mathbb{E}\left[ W_{C,f}^{(r)}(x) W_{C,f}^{(r)}(y) \right] - \mathbb{E}\left[ W_{C,f}^{(r)}(x) \right] \mathbb{E}\left[ W_{C,f}^{(r)}(y) \right].
\]

To facilitate the calculation of these moments over random matrices, we reformulate the enumerator in terms of the ranks of column submatrices of the generator matrix.

\subsection{Random linear codes and the rank-function formulation}

Let $G$ be a random $k \times n$ matrix whose entries are independent and uniformly distributed over $\F_q$, and let $C \subseteq \F_q^n$ be the linear code generated by the rows of $G$. We call $C$ a random $[n, k]$ code. For a subset $I \subseteq E$, let $G_I$ denote the $k \times |I|$ submatrix of $G$ formed by the columns indexed by $I$, and let
\[
	r_I := \operatorname{rank}(G_I) = \dim C^I.
\]

By Proposition \ref{prop:dim_formulas}(i), the dimension of the shortened code on the complement of $I$ is
\[
\dim C(E \setminus I) = k - \dim C^I = k - r_I.
\]
Consequently, the Gaussian binomial coefficient $\qbinom{k - r_I}{r}$ counts the number of $r$-dimensional subcodes of $C$ whose support is contained in $E \setminus I$. This observation suggests packaging these rank statistics into a polynomial weighted by the harmonic values $\widetilde{f}(I)$.

\begin{definition}\label{def:rank_function}
	Let $G$ be a $k \times n$ matrix.
	Let $f \in \Harm_d(n)$ and $r \ge 0$ be an integer. The \emph{harmonic rank function} of $G$ is
	\[
		\mathcal{R}_{G,f}^{(r)}(z) := \sum_{I \subseteq E} \qbinom{k - r_I}{r} \widetilde{f}(I) z^{|I|}.
	\]
\end{definition}

\begin{remark}
The coefficient of $z^{|I|}$ records the evaluation $\widetilde{f}(I)$ of the harmonic function on $I$, weighted by the number of $r$-dimensional subcodes of $C$ whose support is contained in $E \setminus I$. 
For $d=0$, $\mathcal{R}_{G,f}^{(r)}(z)$ coincides with the ordinary rank-function formulation studied by Özen and Tekin \cite{ozen2013}.
\end{remark}

\begin{proposition}\label{prop:rank}
Let $G$ be a full-rank $k \times n$ generator matrix of 
an $[n,k]$ code $C$. Then
\[
\mathcal{R}_{G,f}^{(r)}(z) = (-1)^d \frac{z^d}{(1+z)^d} W_{C,f}^{(r)}(1+z).
\]
\end{proposition}

\begin{proof}
We begin with Lemma \ref{lem:harm_lemma1}, which we multiply by $z^{n-j}$ and sum over all integers $j$:
\[
\underbrace{\sum_{j} z^{n-j} \sum_{m} \qbinom{m}{r} N_{j,f}^m(C)}_{(*)} = \underbrace{\sum_{j} z^{n-j} \sum_{i=d}^{n-d} \binom{n-i-d}{n-j-d} A_{i,f}^{(r)}(C)}_{(**)}.
\]

For the left-hand side $(*)$, we unfold the definition of $N_{j,f}^m(C)$:
\[
(*) = \sum_{j} \sum_{\substack{X \subseteq E \\ |X|=j \\ \dim C(X)=m}} \qbinom{m}{r} \widetilde{f}(X) z^{n-j} = \sum_{X \subseteq E} \qbinom{\dim C(X)}{r} \widetilde{f}(X) z^{n-|X|}.
\]
Set $I = E \setminus X$, so that $|I| = n - |X|$. By Proposition \ref{prop:dim_formulas}(i) and reflection identity (Lemma~\ref{prop:reflection}),
\[
\qbinom{\dim C(X)}{r} = \qbinom{k-r_I}{r}, \qquad \widetilde{f}(X) = \widetilde{f}(E \setminus I) = (-1)^d \widetilde{f}(I).
\]
Substituting these into the sum gives
\[
(*) = (-1)^d \sum_{I \subseteq E} \qbinom{k-r_I}{r} \widetilde{f}(I) z^{|I|} = (-1)^d \mathcal{R}_{G,f}^{(r)}(z).
\]

For the right-hand side $(**)$, we interchange the sums and evaluate the inner sum over $j$ using the binomial identity
\[
\sum_{j} \binom{n-i-d}{n-j-d} z^{n-j} = z^d (1+z)^{n-i-d}.
\]
Thus
\[
(**) = z^d \sum_{i=d}^{n-d} A_{i,f}^{(r)}(C) (1+z)^{n-i-d} = \frac{z^d}{(1+z)^d} \sum_{i=d}^{n-d} A_{i,f}^{(r)}(C) (1+z)^{n-i} = \frac{z^d}{(1+z)^d} W_{C,f}^{(r)}(1+z).
\]
Equating $(*)$ and $(**)$ and dividing by $(-1)^d$ yields the claimed identity.
\end{proof}

\begin{remark}\label{rem:R_divisible}
Since $\widetilde{f}(I) = 0$ for $|I| < d$, the summands in the definition of $\mathcal{R}_{G,f}^{(r)}(z)$ vanish whenever $|I| < d$. This implies that the polynomial $\mathcal{R}_{G,f}^{(r)}(z)$ is always divisible by $z^d$, and we may write:
\[
\mathcal{R}_{G,f}^{(r)}(z) = z^d \sum_{I \subseteq E} \qbinom{k-r_I}{r} \widetilde{f}(I) z^{|I|-d}.
\]
\end{remark}

\begin{remark}
For $d=0$, Proposition \ref{prop:rank} reduces to $\mathcal{R}_{G}^{(r)}(z) = W_{C}^{(r)}(1+z)$, which recovers \cite[Eq.~(1)]{ozen2013}.
\end{remark}

\begin{corollary}\label{cor:W_from_R}
The harmonic $r$-th higher weight enumerator $W_{C,f}^{(r)}(z)$ can be expressed directly in terms of the harmonic rank function $\mathcal{R}_{G,f}^{(r)}(z)$ via the identity:
\[
W_{C,f}^{(r)}(z) = (-1)^d \frac{z^d}{(z-1)^d} \mathcal{R}_{G,f}^{(r)}(z-1).
\]
\end{corollary}

\begin{proof}
By Proposition \ref{prop:rank}, we have the polynomial identity
\[
\mathcal{R}_{G,f}^{(r)}(w) = (-1)^d \frac{w^d}{(1+w)^d} W_{C,f}^{(r)}(1+w).
\]
Substituting $w = z - 1$, we note that $1 + w = z$. This yields
\[
\mathcal{R}_{G,f}^{(r)}(z-1) = (-1)^d \frac{(z-1)^d}{z^d} W_{C,f}^{(r)}(z).
\]
Multiplying both sides by $(-1)^d \frac{z^d}{(z-1)^d}$ gives the claimed identity:
\[
W_{C,f}^{(r)}(z) = (-1)^d \frac{z^d}{(z-1)^d} \mathcal{R}_{G,f}^{(r)}(z-1). \qedhere
\]
\end{proof}

\subsection{The Expectation (First Moment)}

Özen and Tekin \cite[Theorem 3.1]{ozen2013} shows that for a uniformly random matrix, the expected value of the rank coefficient depends {only} on the size of the set $I$, evaluating exactly to:
\[
	\mathbb{E}\left[ \qbinom{k - r_I}{r} \right] 
	= 
	\qbinom{k}{r} q^{-r|I|}.
\]
 
Since $\widetilde{f}(I)$ and $z^{|I|}$ are deterministic,
and by linearity of expectation, 
we have
\begin{align*}
	\mathbb{E}\left[ \mathcal{R}_{G,f}^{(r)}(z) \right] 
	& = 
	\sum_{I \subseteq E} 
	\widetilde{f}(I) 
	\mathbb{E}\left[ \qbinom{k - r_I}{r} \right] 
	z^{|I|}\\
	& =
	\qbinom{k}{r} 
	\sum_{I \subseteq E} 
	\widetilde{f}(I) (q^{-r}z)^{|I|}\\
	& =
	\qbinom{k}{r} 
	\sum_{I \subseteq E}
	\widetilde{f}(I)\, t^{|I|}, \qquad \text{where } t = q^{-r}z.
\end{align*}

\noindent
Now by definition~(\ref{Equ:tildef}) and swapping the summation, we get
\[
	\sum_{I \subseteq E} \widetilde{f}(I)\, t^{|I|} 
	= 
	\sum_{I \subseteq E} \sum_{\substack{Z \subseteq I \\ |Z|=d}} 
	f(Z)\, t^{|I|} = \sum_{Z \in \binom{E}{d}} f(Z) \sum_{\substack{I \subseteq E \\ Z \subseteq I}} t^{|I|}.
\]
For a fixed $d$-subset $Z$, the inner sum over all supersets $I$ evaluates to $t^d (1+t)^{n-d}$. Thus
\[
	\sum_{I \subseteq E} \widetilde{f}(I)\, t^{|I|} 
	= 
	t^d (1+t)^{n-d} 
	\sum_{Z \in \binom{E}{d}} f(Z) 
	= 
	t^d (1+t)^{n-d} \widetilde{f}(E).
\]
Since $\widetilde{f}(E) = 0$ 
for any $f\in\textrm{Harm}_{d}(n)$ with $d \neq 0$,
the entire sum vanishes.
This yields our first moment theorem:

\begin{theorem}\label{thm:expectation}
Let $C$ be a random $[n, k]$ code, and let $f \in \Harm_d(n)$ 
with $d \neq 0$. The expectation of the harmonic rank function is zero.
That is,
\[ \mathbb{E}\left[ \mathcal{R}_{G,f}^{(r)}(z) \right] = 0. \]
\end{theorem}

\subsection{The Covariance (Second Moment)}

Since $\mathbb{E}[\mathcal{R}_{G,f}^{(r)}(z)] = 0$ (Theorem~\ref{thm:expectation}),
the covariance 
\[
\Cov\left( \mathcal{R}_{G,f}^{(r)}(x),\, \mathcal{R}_{G,f}^{(r)}(y) \right) = \mathbb{E}\left[ \mathcal{R}_{G,f}^{(r)}(x)\, \mathcal{R}_{G,f}^{(r)}(y) \right].
\]

\begin{theorem}[Covariance Formula]\label{thm:cov}
Let $f \in \Harm_d(n)$ with $d \ge 1$. Set
\[
\alpha = q^{-r}x, \qquad \beta = q^{-r}y, \qquad c_m = q^{r-m}.
\]
Define the \emph{pair correlation} of $f$ by
\[
\mathcal{P}_s := \sum_{\substack{Z, W \in \binom{E}{d} \\ |Z \cap W| = s}} f(Z) f(W), \qquad s = 0, 1, \dots, d.
\]
Because $\sum_{s=0}^d \mathcal{P}_s = \widetilde{f}(E)^2 = 0$, we have $\mathcal{P}_d = -\sum_{s=0}^{d-1} \mathcal{P}_s$. The covariance of the harmonic rank function is
\begin{multline}\label{eq:cov}
    \Cov\left( \mathcal{R}_{G,f}^{(r)}(x),\, \mathcal{R}_{G,f}^{(r)}(y) \right) = \qbinom{k}{r} \alpha^d \beta^d \sum_{m=0}^{r} q^{m^2} \qbinom{k-r}{m} \qbinom{r}{m} \\
    \times \sum_{s=0}^{d} \mathcal{P}_s c_m^s (1 + \alpha c_m)^{d-s} (1 + \beta c_m)^{d-s} (1 + \alpha + \beta + \alpha \beta c_m)^{n-2d+s}.
\end{multline}

\end{theorem}

\begin{proof}
By bilinearity of covariance and the fact that $\mathbb{E}\left[\mathcal{R}_{G,f}^{(r)}(z)\right] = 0$,
\begin{equation}\label{eq:cov-step1}
\Cov\left( \mathcal{R}_{G,f}^{(r)}(x), \mathcal{R}_{G,f}^{(r)}(y) \right) = \sum_{I,J \subseteq E} \widetilde{f}(I) \widetilde{f}(J) x^{|I|} y^{|J|} \Cov\left( \qbinom{k-r_I}{r}, \qbinom{k-r_J}{r} \right).
\end{equation}

\medskip
\noindent\textit{Step 1: Apply the rank-covariance formula.}
\smallskip

For a uniformly random matrix, Özen and Tekin \cite[Theorem 3.2]{ozen2013} proved that
\begin{multline}\label{eq:ozen-cov}
\Cov\left( \qbinom{k-r_I}{r}, \qbinom{k-r_J}{r} \right) \\
= q^{-r(|I|+|J|)} \qbinom{k}{r} \sum_{m=0}^r q^{m^2} \qbinom{k-r}{m} \qbinom{r}{m} \left( q^{|I \cap J|(r-m)} - 1 \right).
\end{multline}
Substituting \eqref{eq:ozen-cov} into \eqref{eq:cov-step1} and using $\alpha = q^{-r}x$, $\beta = q^{-r}y$, and $c_m = q^{r-m}$, we obtain
\[
\Cov\left( \mathcal{R}_{G,f}^{(r)}(x), \mathcal{R}_{G,f}^{(r)}(y) \right) = \qbinom{k}{r} \sum_{m=0}^r q^{m^2} \qbinom{k-r}{m} \qbinom{r}{m} \Lambda_m,
\]
where
\[
\Lambda_m = \sum_{I,J \subseteq E} \widetilde{f}(I) \widetilde{f}(J) \alpha^{|I|} \beta^{|J|} \left( c_m^{|I \cap J|} - 1 \right).
\]

\medskip
\noindent\textit{Step 2: Split $\Lambda_m$ and remove the product term.}
\smallskip

Expanding the bracket gives
\[
\Lambda_m = \underbrace{\sum_{I,J} \widetilde{f}(I) \widetilde{f}(J) \alpha^{|I|} \beta^{|J|} c_m^{|I \cap J|}}_{=:\,\Psi_m} - \underbrace{\left( \sum_{I} \widetilde{f}(I) \alpha^{|I|} \right) \left( \sum_{J} \widetilde{f}(J) \beta^{|J|} \right)}_{=:\,\Pi}.
\]
The product term $\Pi$ vanishes for $d \ge 1$, because
\[
\sum_{I} \widetilde{f}(I) \alpha^{|I|} = \alpha^d (1+\alpha)^{n-d} \widetilde{f}(E) = 0,
\]
and similarly for the sum over $J$. Hence $\Lambda_m = \Psi_m$.

\medskip
\noindent\textit{Step 3: Evaluate $\Psi_m$ by expanding $\widetilde{f}$.}
\smallskip

Writing $\widetilde{f}(I) = \sum_{Z \in \binom{I}{d}} f(Z)$ and $\widetilde{f}(J) = \sum_{W \in \binom{J}{d}} f(W)$, and swapping the order of summation, we get
\[
\Psi_m = \sum_{Z,W \in \binom{E}{d}} f(Z) f(W) G_{Z,W}(\alpha,\beta,c_m),
\]
where
\[
G_{Z,W}(\alpha,\beta,c_m) := \sum_{\substack{I \supseteq Z \\ J \supseteq W}} \alpha^{|I|} \beta^{|J|} c_m^{|I \cap J|}.
\]

Fix $Z,W$ with $|Z \cap W| = s$. Decompose $E$ into four disjoint regions:
\[
\begin{array}{c|c}
\text{Region} & \text{Size} \\ \hline
Z \cap W & s \\
Z \setminus W & d-s \\
W \setminus Z & d-s \\
E \setminus (Z \cup W) & n-2d+s
\end{array}
\]
For each element $e$, the contribution to $G_{Z,W}$ from the choices $(e \in I?, e \in J?)$ is:
\[
\begin{array}{c|c|c}
\text{Region} & \text{Constraint} & \text{Factor} \\ \hline
Z \cap W & e \in I,\; e \in J & \alpha\beta c_m \\[2pt]
Z \setminus W & e \in I,\; e \notin J \text{ optional} & \alpha(1 + \beta c_m) \\[2pt]
W \setminus Z & e \notin I \text{ optional},\; e \in J & \beta(1 + \alpha c_m) \\[2pt]
E \setminus (Z \cup W) & \text{both optional} & 1 + \alpha + \beta + \alpha\beta c_m
\end{array}
\]
Multiplying the independent contributions over all elements yields
\[
G_{Z,W}(\alpha,\beta,c_m) = \alpha^d \beta^d c_m^s (1 + \alpha c_m)^{d-s} (1 + \beta c_m)^{d-s} (1 + \alpha + \beta + \alpha\beta c_m)^{n-2d+s}.
\]
Since $G_{Z,W}$ depends only on $s = |Z \cap W|$, we may group the pairs $(Z,W)$ by their intersection size:
\[
\Psi_m = \alpha^d \beta^d \sum_{s=0}^d \mathcal{P}_s c_m^s (1 + \alpha c_m)^{d-s} (1 + \beta c_m)^{d-s} (1 + \alpha + \beta + \alpha\beta c_m)^{n-2d+s}.
\]
Substituting $\Lambda_m = \Psi_m$ gives formula \eqref{eq:cov}.

\medskip
\noindent\textit{Step 4: The $m=r$ term vanishes.}
\smallskip

When $m=r$, we have $c_r = q^{r-r} = 1$. Then
\[
1 + \alpha + \beta + \alpha\beta c_r = (1+\alpha)(1+\beta),
\]
and the $m=r$ term evaluates to
\[
\alpha^d \beta^d (1+\alpha)^{n-d}(1+\beta)^{n-d} \sum_{s=0}^d \mathcal{P}_s = \alpha^d \beta^d (1+\alpha)^{n-d}(1+\beta)^{n-d} \widetilde{f}(E)^2 = 0.
\]
Thus only the terms with $m < r$ contribute non-trivially.
\end{proof}

\section{Examples for Moments of Random Codes}

\begin{example}\label{ex:covariance}
We explicitly calculate the covariance for a small random code. Let
\[
q=2, \qquad n=2, \qquad k=1, \qquad r=1, \qquad d=1.
\]
Define $f \in \Harm_1(2)$ by
\[
f(\{1\}) = 1, \qquad f(\{2\}) = -1.
\]
The set-extension values are
\[
\widetilde{f}(\{1\}) = 1, \qquad \widetilde{f}(\{2\}) = -1, \qquad \widetilde{f}(\emptyset) = \widetilde{f}(E) = 0.
\]

There are $2^2 = 4$ possible $1 \times 2$ generator matrices $G = [c_1, c_2]$ over $\F_2$, each chosen with probability $1/4$. For $r=1$, the rank coefficient {${\qbinom{1-r_I}{1}}$} equals $1$ if $r_I=0$ (the column is zero) and $0$ if $r_I=1$. Let
\[
z_i = 
\begin{cases}
1 & \text{if } c_i = 0, \\[2pt]
0 & \text{if } c_i = 1.
\end{cases}
\]
Then the harmonic rank function for a specific matrix is
\[
\mathcal{R}_{G,f}^{(1)}(y) = z_1 \widetilde{f}(\{1\}) y + z_2 \widetilde{f}(\{2\}) y = (z_1 - z_2)y.
\]

For brevity, we will write $\mathcal{R}(z)$ to denote $\mathcal{R}_{G,f}^{(1)}(z)$ for the remainder of this example.
\medskip

We evaluate $\mathcal{R}$ for the four matrices:
\begin{itemize}
    \item $G_1 = [0, 0]$: $(z_1, z_2) = (1, 1)$, so
    \[
    \mathcal{R}_1 = (1-1)y = 0.
    \]
    
    \item $G_2 = [1, 0]$: $(z_1, z_2) = (0, 1)$, so
    \[
    \mathcal{R}_2 = (0-1)y = -y.
    \]
    
    \item $G_3 = [0, 1]$: $(z_1, z_2) = (1, 0)$, so
    \[
    \mathcal{R}_3 = (1-0)y = y.
    \]
    
    \item $G_4 = [1, 1]$: $(z_1, z_2) = (0, 0)$, so
    \[
    \mathcal{R}_4 = (0-0)y = 0.
    \]
\end{itemize}

The expectation is therefore
\[
\mathbb{E}[\mathcal{R}] = \frac{1}{4}\bigl(0 - y + y + 0\bigr) = 0,
\]
confirming Theorem \ref{thm:expectation}.

\medskip
To find the covariance, we compute the expectation of the product $\mathcal{R}(x)\mathcal{R}(y)$. Since
\[
\mathcal{R}(x) = (z_1 - z_2)x \quad\text{and}\quad \mathcal{R}(y) = (z_1 - z_2)y,
\]
we have
\[
\mathcal{R}(x)\mathcal{R}(y) = (z_1 - z_2)^2 xy.
\]
Because $z_i \in \{0,1\}$, it follows that $z_i^2 = z_i$, and therefore
\[
(z_1 - z_2)^2 = z_1 - 2z_1 z_2 + z_2.
\]
Evaluating this expression for the four matrices gives the values $0, 1, 1, 0$, respectively. Hence
\begin{align*}
\mathrm{Cov}\bigl(\mathcal{R}(x), \mathcal{R}(y)\bigr)
&= \mathbb{E}\bigl[\mathcal{R}(x)\mathcal{R}(y)\bigr] \\[4pt]
&= \frac{1}{4}\bigl(0 + xy + xy + 0\bigr) \\[4pt]
&= \frac{1}{2}xy. 
\end{align*}

\end{example}

\begin{example}\label{ex:covariance_d3}
We calculate the moments for $d=3$ to demonstrate that while the expectation vanishes, the covariance is strictly non-zero. Let
\[
q=2, \qquad n=32, \qquad k=1, \qquad r=1, \qquad d=3.
\]

Let $E$ be a set of $32$ elements, partitioned into $4$ disjoint subsets
\[
E_1,\; E_2,\; E_3,\; E_4,
\]
each of size $8$. For each $E_i$, let it be further partitioned into two sets of size $4$:
\[
E_i = A_i \cup B_i.
\]
We define $f \in \Harm_3(32)$ as a sum of functions $f_i$ supported on $E_i$. For a $3$-subset $Z \subseteq E_i$, we define $f_i(Z)$ based on $x = |Z \cap A_i|$ as follows:
\[
f_i(Z) = 
\begin{cases} 
\phantom{-}2v_i & \text{if } x = 0, \\[4pt]
-v_i & \text{if } x = 1, \\[4pt]
\phantom{-}v_i & \text{if } x = 2, \\[4pt]
-2v_i & \text{if } x = 3,
\end{cases}
\]
where $(v_1, v_2, v_3, v_4) = (a, b, c, d)$. For any $3$-subset $Z$ that intersects more than one $E_i$, we set $f(Z) = 0$.

\medskip
\noindent\textit{Step 1: Verifying that $f$ is harmonic.}
\smallskip

For $f$ to be harmonic, we need
\[
\sum_{e \in E \setminus Y} f(Y \cup \{e\}) = 0
\qquad\text{for all } Y \in \binom{E}{2}.
\]
If $Y$ intersects multiple $E_i$, the sum is trivially $0$ because all $Y \cup \{e\}$ will also intersect multiple $E_i$. If $Y \subset E_i$, the sum only involves $e \in E_i$. By the construction of $f_i$, the sum over $e \in A_i$ and $e \in B_i$ perfectly cancels out:
\begin{itemize}
    \item If $Y \subset A_i$ (so $x=2$): adding $e \in A_i$ gives $2$ choices with $x=3 \to -2v_i$, and adding $e \in B_i$ gives $4$ choices with $x=2 \to v_i$. The sum is
    \[
    2(-2v_i) + 4(v_i) = 0.
    \]
    
    \item If $Y \subset B_i$ (so $x=0$): adding $e \in A_i$ gives $4$ choices with $x=1 \to -v_i$, and adding $e \in B_i$ gives $2$ choices with $x=0 \to 2v_i$. The sum is
    \[
    4(-v_i) + 2(2v_i) = 0.
    \]
    
    \item If $Y$ has $1$ element in $A_i$ and $1$ in $B_i$ (so $x=1$): adding $e \in A_i$ gives $3$ choices with $x=2 \to v_i$, and adding $e \in B_i$ gives $3$ choices with $x=1 \to -v_i$. The sum is
    \[
    3(v_i) + 3(-v_i) = 0.
    \]
\end{itemize}
Thus, $f$ is harmonic.

\medskip
\noindent\textit{Step 2: Derivation of the expectation (first moment).}
\smallskip

By Theorem \ref{thm:expectation}, the expectation is structurally zero. We derive this explicitly from the definition. For $k=1$ and $r=1$, the expected rank coefficient is
\[
\mathbb{E}\!\left[{\qbinom{1-r_I}{1}}\right] = 2^{-|I|}.
\]
Thus
\[
\mathbb{E}[\mathcal{R}(z)] = \sum_{I \subseteq E} \widetilde{f}(I) \left(\frac{z}{2}\right)^{|I|}.
\]
To evaluate this sum, we expand
\[
\widetilde{f}(I) = \sum_{\substack{Z \subseteq I \\ |Z|=3}} f(Z)
\]
and swap the order of summation. Let $t = z/2$. Then
\begin{align*}
\sum_{I \subseteq E} \widetilde{f}(I)\, t^{|I|}
&= \sum_{I \subseteq E} \sum_{\substack{Z \subseteq I \\ |Z|=3}} f(Z)\, t^{|I|} \\[4pt]
&= \sum_{Z \in \binom{E}{3}} f(Z) \sum_{\substack{I \subseteq E \\ Z \subseteq I}} t^{|I|}.
\end{align*}
For a fixed $3$-element set $Z$, the inner sum runs over all supersets $I \supseteq Z$. Since $|I| = 3 + |I \setminus Z|$ and the $n-3 = 29$ elements outside $Z$ are free, we obtain
\[
\sum_{\substack{I \supseteq Z}} t^{|I|} = t^3 (1+t)^{29}.
\]
Pulling this out gives
\[
\sum_{I \subseteq E} \widetilde{f}(I)\, t^{|I|}
= t^3 (1+t)^{29} \sum_{Z \in \binom{E}{3}} f(Z).
\]
Recall that
\[
\sum_{Z \in \binom{E}{3}} f(Z) = \widetilde{f}(E).
\]
Because the $E_i$ are disjoint, the sum splits into $4$ independent sums over $E_i$. For a single $E_i$, the number of $3$-subsets with $|Z \cap A_i| = x$ is $\binom{4}{x}\binom{4}{3-x}$. Hence
\begin{align*}
\widetilde{f}(E_i)
&= \binom{4}{0}\binom{4}{3}(2v_i) + \binom{4}{1}\binom{4}{2}(-v_i) + \binom{4}{2}\binom{4}{1}(v_i) + \binom{4}{3}\binom{4}{0}(-2v_i) \\[4pt]
&= 4(2v_i) + 24(-v_i) + 24(v_i) + 4(-2v_i) \\[4pt]
&= 8v_i - 24v_i + 24v_i - 8v_i \\[4pt]
&= 0.
\end{align*}
Summing over all $i$, we get
\[
\widetilde{f}(E) = 0.
\]
Substituting this back into our expectation formula yields
\[
\mathbb{E}[\mathcal{R}(z)]
= \left(\frac{z}{2}\right)^{\!3} \left(1+\frac{z}{2}\right)^{\!29} \cdot 0
= 0.
\]
This explicitly shows that the expectation vanishes because the sum of the harmonic function over all possible $3$-subsets is zero.

\medskip
\noindent\textit{Step 3: The covariance is non-zero.}
\smallskip

\noindent By Theorem \ref{thm:cov}, since $k=1$ and $r=1$, the $m=1$ term vanishes, leaving only $m=0$. Let
\[
c_0 = q^{r-0} = 2, \qquad \alpha = \frac{x}{2}, \qquad \beta = \frac{y}{2}.
\]
The covariance formula becomes
\[
\Cov
= {\qbinom{1}{1}}\, \alpha^3 \beta^3
\sum_{s=0}^3 \mathcal{P}_s\, 2^s (1+2\alpha)^{3-s} (1+2\beta)^{3-s} (1+\alpha+\beta+2\alpha\beta)^{32-6+s}.
\]

\noindent Because $f$ is a sum of disjointly supported functions $f_i$, the pair correlation $\mathcal{P}_s = \sum_{Z,W} f(Z) f(W)$ splits into independent sums over each $E_i$. Note that cross-block pairs (where $Z \subseteq E_i$ and $W \subseteq E_j$ with $i \neq j$) only contribute to $s=0$, and their total sum is proportional to $\widetilde{f}_i(E_i)\widetilde{f}_j(E_j) = 0$, so these cross terms vanish entirely.

\noindent To evaluate $\mathcal{P}_s^{(i)}$ for a single block $E_i$, we use a generating function. Consider the sum over all ordered pairs $(Z, W) \in \binom{E_i}{3}^2$ weighted by $t^{|Z \cap W|}$:
\[ \sum_{s=0}^3 \mathcal{P}_s^{(i)} t^s = \sum_{Z,W \subseteq E_i} f_i(Z) f_i(W) t^{|Z \cap W|}. \]
We apply the set-theoretic identity $t^{|Z \cap W|} = \sum_{U \subseteq Z \cap W} (t-1)^{|U|}$ and swap the order of summation:
\[ \sum_{s=0}^3 \mathcal{P}_s^{(i)} t^s = \sum_{U \subseteq E_i} (t-1)^{|U|} \left( \sum_{\substack{Z \subseteq E_i \\ Z \supseteq U}} f_i(Z) \right) \left( \sum_{\substack{W \subseteq E_i \\ W \supseteq U}} f_i(W) \right) = \sum_{U \subseteq E_i} (t-1)^{|U|} \left( \sum_{\substack{Z \subseteq E_i \\ Z \supseteq U}} f_i(Z) \right)^2. \]
Because $f_i$ is harmonic of degree 3, the inner sum vanishes for all $|U| < 3$. For $|U| = 3$, the only 3-subset $Z \subseteq E_i$ containing $U$ is $Z = U$ itself. Therefore, the sum collapses entirely to the $|U|=3$ terms:
\[ \sum_{s=0}^3 \mathcal{P}_s^{(i)} t^s = (t-1)^3 \sum_{U \in \binom{E_i}{3}} f_i(U)^2 = (t-1)^3 \mathcal{P}_3^{(i)}. \]
We calculate $\mathcal{P}_3^{(i)} = \sum_{U} f_i(U)^2$ using the definition of $f_i$:
\begin{align*}
    \mathcal{P}_3^{(i)} &= \binom{4}{0}\binom{4}{3}(2v_i)^2 + \binom{4}{1}\binom{4}{2}(-v_i)^2 + \binom{4}{2}\binom{4}{1}(v_i)^2 + \binom{4}{3}\binom{4}{0}(-2v_i)^2 \\
    &= 4(4v_i^2) + 24(v_i^2) + 24(v_i^2) + 4(4v_i^2) = 80v_i^2.
\end{align*}
Finally, we expand $(t-1)^3 = \sum_{s=0}^3 \binom{3}{s} t^s (-1)^{3-s}$. Equating the coefficients of $t^s$ yields the explicit formula:
\[ \mathcal{P}_s^{(i)} = 80 v_i^2 \binom{3}{s} (-1)^{3-s} = 80 v_i^2 \binom{3}{s} (-1)^{s+1}. \]

\noindent Substituting these into the sum for a single block gives
\[
\sum_{s=0}^3 \mathcal{P}_s^{(i)}\, 2^s (1+2\alpha)^{3-s} (1+2\beta)^{3-s} (1+\alpha+\beta+2\alpha\beta)^{26+s}.
\]
Now set
\[
A = 1+\alpha+\beta+2\alpha\beta, \qquad B = (1+2\alpha)(1+2\beta) = 1+2\alpha+2\beta+4\alpha\beta.
\]
Notice that $B = 2A - 1$. The sum becomes
\[
80 v_i^2 \Big[ -B^3 A^{26} + 6 A^{27} B^2 - 12 A^{28} B + 8 A^{29} \Big]
= 80 v_i^2 A^{26} \Big[ -B^3 + 6 A B^2 - 12 A^2 B + 8 A^3 \Big].
\]
The bracket is exactly $-(B - 2A)^3$. Since $B = 2A - 1$, we have
\[
B - 2A = -1,
\]
so
\[
-(B-2A)^3 = -(-1)^3 = 1.
\]
The sum simplifies beautifully to
\[
80 v_i^2 A^{26}.
\]
Summing over the $4$ independent sets $E_i$, the total covariance evaluates cleanly to
\[
\Cov
= 80(a^2 + b^2 + c^2 + d^2)\, \alpha^3 \beta^3 A^{26}.
\]
In terms of $x$ and $y$, this is
\[
\Cov
= 80(a^2 + b^2 + c^2 + d^2) \left(\frac{x}{2}\right)^{\!3} \left(\frac{y}{2}\right)^{\!3} \left(1+\frac{x}{2}+\frac{y}{2}+\frac{xy}{2}\right)^{\!26}.
\]

\noindent This demonstrates explicitly for $d=3$ that while the harmonic enumerator is a centered variable (mean zero), it possesses a strictly non-trivial second moment that evaluates cleanly via Theorem \ref{thm:cov}. As long as $a, b, c, d$ are not all zero, the covariance is strictly non-zero.
\end{example}

\section{Conclusion}

\noindent We have presented a combinatorial proof of the MacWilliams identity for harmonic higher weight distributions by adapting Simonis' method to the harmonic setting via reflection identity. Furthermore, we investigated the statistical moments of the harmonic enumerator for random linear codes. Natural continuations of this work include the analysis of higher moments beyond the second, the asymptotic behavior of the covariance as the code length grows, and numerical experiments for larger parameter regimes.

\section*{Acknowledgements}
This work was supported by SUST Research Centre (PS/2025/1/19).

\section*{Data availability statement}
The data that support the findings of this study are available from
the corresponding author.

\end{document}